\documentclass[psamsfont,a4paper,12pt]{amsart}

\usepackage[english]{babel}   

\usepackage{setspace}     
\usepackage{enumerate} 	
\usepackage{afterpage} 	
\usepackage[usenames,dvipsnames]{color}
\usepackage{graphicx}	

\usepackage{amsmath} 	
\usepackage{amssymb}	
\usepackage{amsfonts} 	
\usepackage{latexsym} 	
\usepackage[latin1]{inputenc} 

\usepackage{amsthm} 	
\usepackage{stackrel}       
\usepackage{amscd} 	

\usepackage{tabularx}	
\usepackage{longtable}	

\usepackage{verbatim} 
\usepackage{blindtext} 

\usepackage{lipsum}

\allowdisplaybreaks

\newcommand{\C}{\mathbb{C}}

\newcommand{\R}{\mathbb{R}} 
 
\newcommand{\N}{\mathbb{N}} 
\newcommand{\Z}{\mathbb{Z}} 
\newcommand{\Q}{\mathbb{Q}}

\newcommand{\HH}{\mathbb{H}}

\renewcommand{\to}{\longrightarrow}

\newtheorem{Thm}{Theorem}[section]		

\newtheorem{Prop}[Thm]{Proposition}

\newtheorem*{Ex}{Example}

\theoremstyle{definition}
\newtheorem{Definition}[Thm]{Definition}

\theoremstyle{remark}
 
\newtheorem*{rmk}{Remark}

\newtheorem{ind}[]{{\rm\it Indice}}

\usepackage{multicol}

\usepackage{tikz}
\usetikzlibrary{arrows}

\usepackage{paralist}
\usepackage{cite}

\usepackage{breqn}   

\title{A note on Schwartz functions and modular forms}
\author[Rolen]{Larry Rolen}
\author[Wagner]{Ian Wagner}

\begin{document}
\numberwithin{equation}{section}

\begin{abstract}
We generalize the recent work of Viazovska by constructing infinite families of Schwartz functions, suitable for Cohn-Elkies style linear programming bounds, using quasi-modular and modular forms.  In particular for dimensions $d \equiv 0 \pmod{8}$ we give the constructions that lead to the best sphere packing upper bounds via modular forms.  In dimension $8$ and $24$ these exactly match the functions constructed by Viazovska and Cohn, Kumar, Miller, Radchenko, and Viazovska which resolved the sphere packing problem in those dimensions.
\end{abstract}

\maketitle

\section{Introduction and statement of results}
The sphere packing problem started in 1611 when Kepler asked for the best way to stack cannonballs in a crate.  This is the dimension $3$ case, but more generally one can ask what proportion of $\R^d$ can be covered with non-overlapping congruent balls.  To be more precise, if $\mathcal{P}$ is a packing, then the \textit{finite density} of $\mathcal{P}$, truncated at some radius $r$, is 
\begin{equation*}
\Delta_{\mathcal{P}}(r) := \frac{\text{Vol}(\mathcal{P} \cap B_{d}(0, r))}{\text{Vol}(B_{d}(0,r))}.
\end{equation*}  
The \textit{density} of $\mathcal{P}$ is then $\Delta_{\mathcal{P}}:= \limsup_{r \to \infty} \Delta_{\mathcal{P}}(r)$ and the \textit{sphere packing constant} is 
\begin{equation}
\Delta_{d} := \sup_{\mathcal{P} \subset \R^d} \Delta_{\mathcal{P}}.
\end{equation}
The sphere packing problem in dimension $d$ is to determine $\Delta_{d}$ for a particular dimension $d$.  Clearly we have $\Delta_{1} =1$, and in 1892 Thue \cite{Th} exactly computed $\Delta_{2} \approx 0.9068$ by proving the hexagonal packing corresponding to the $A_{2}$ lattice is optimal for $d=2$.  It wasn't until 1998 that Kepler's original question was answered; Hales \cite{H} showed that $\Delta_{3} = \frac{\pi}{\sqrt{18}} \approx 0.7405$.  Recently a breakthrough was made by Cohn and Elkies that showed solving the sphere packing problem in dimensions $8$ and $24$ was within reach. Essentially, they showed that the proof could be reduced to the construction of special functions satisfying linear programming bounds, where one needs to control the size of the function and its Fourier transform simultaneously. To recall their results, we define the \textit{Fourier transform} of an ${\rm{L}}^{1}$ function $f: \R^{d} \to \C$ by
\begin{equation}
\mathcal{F}(f)(y) = \widehat{f}(y) := \int_{\R^{d}} f(x) e^{-2 \pi i \langle x, y \rangle} dx \qquad y \in \R^{d},
\end{equation}
where $\langle x, y \rangle$ is the standard scalar product in $\R^{d}$.  A classic example is that the Fourier transform of the Gaussian is another Gaussian
\begin{equation*}
\mathcal{F} \left( e^{-\alpha x^2} \right) = \left( \frac{\pi}{\alpha} \right)^{\frac{d}{2}} e^{-\frac{\pi^2 y^2}{\alpha}}.
\end{equation*}

Given a lattice $\Lambda$ with shortest nonzero vector of length $r_{0}$, the density of the corresponding lattice packing is
\begin{equation*}
\Delta_{d}:=\frac{\pi^{d/2}}{\Gamma \left(\frac{d}{2} +1 \right)} \left(\frac{r_{0}}{2} \right)^{d} \frac{1}{|\Lambda|}.
\end{equation*}
A function $f(x)$ is a \textit{Schwartz function} if $f$ and all of its derivatives decay to zero faster than any inverse power of $x$.  In \cite{CE} Cohn and Elkies show that if $\Lambda$ is a self dual lattice with shortest nonzero vector of length $r_{0}$ and $f: \R^{d} \to \R$ is a Schwartz function satisfying the following two conditions:
\begin{enumerate}
\item $f(x) \leq 0$ for all $|x| \geq r_{0}$,
\item $\widehat{f}(x) \geq 0$ for all $x \in \R^d$,
\end{enumerate}
then 
\begin{equation*}
\Delta_{d} \leq \frac{\pi^{d/2}}{\Gamma \left(\frac{d}{2} +1 \right)} \left(\frac{r_{0}}{2} \right)^{d} \frac{f(0)}{\widehat{f}(0)}.
\end{equation*}
This kind of result is known as a \textit{linear programming bound}.  Cohn and Elkies constructed functions for $4 \leq d \leq 36$ which, when combined with their theorem, led to the  best known upper bounds for sphere packing in those dimensions.  In particular, they showed that the upper bound in dimensions $8$ and $24$ was extremely close to the known lower bound, which provided evidence that there existed functions which would resolve the sphere packing problem in those dimensions.  

In 2017 \cite{V} Viazovska explicitly constructed such a function for $d=8$ using special modular forms and quasi-modular forms which implied that the $E_{8}$ lattice packing is optimal in $8$ dimensions.  Her methods were quickly modified by Cohn, Kumar, Miller, Radchenko, and Viazovska to show that the Leech lattice packing is optimal for $d=24$ \cite{ckmrv}.  The main ideas behind the proofs of these theorems was to split the problem of constructing $f$ into constructing a function $f_{+}$ which is a $+1$ eigenfunction for the Fourier transform and $f_{-}$ which is a $-1$ eigenfunction for the Fourier transform.  Letting $f$ be a linear combination of these two functions allows control over the necessary inequalities.  The Poisson Summation Formula also tells us that in order for the function $f$ to resolve the sphere packing problem in a given dimension it also needs to have zeros of specific orders at specific points.  To be precise, if $r_{0}$ is the shortest vector length in a lattice packing, then $f(x)$ must have double zeros at all lattice points $|x| > r_{0}$ and a simple zero when $|x|=r_{0}$.

For other dimensions not much is known about the sphere packing densities.  There are conjectures for optimal packings in small dimensions, but few results have been proven.  The best known lower bound is due to Venkatesh \cite{Ve} and gives
\begin{equation*}
\Delta_{d} \geq \frac{e^{-\gamma}}{2} \log \log d \cdot 2^{-d},
\end{equation*}
but is only true for a sparse sequence of dimensions.  The best known upper bound has not been improved since 1978 when Kabatiansky and Levenshtein \cite{KL} proved
\begin{equation*}
\Delta_{d} \leq 2^{-0.599d}.
\end{equation*}
There are problems related to sphere packing for which Viazovska's construction may be useful.  In particular, sphere packing is just a special case of an energy optimization problem.  Cohn, Kumar, Miller, Radchenko, and Viazovska \cite{ckmrv2} recently used related methods to prove that the $E_{8}$ lattice and the Leech lattice are universally optimal (see \cite{CK} for definition) in $8$ and $24$ dimensions respectively.   

There is hope that the techniques developed to prove the sphere packing and energy optimization results given above can be used to attack related problems.  Here we generalize Viazovska's result and construct Schwartz functions using special quasi-modular and modular forms.  We can completely determine the zeros of these functions and how they behave under the Fourier transform.
\begin{Thm} \label{Schwartz}
For each dimension $d \equiv 0 \pmod{4}$ there exists a radial Schwartz function $f: \R^{d} \to \R$ and an $n \in \N$ such that
\begin{itemize}
\item $f(x) = f_{+}(x) + f_{-}(x)$ and $\widehat{f}(x) = (-i)^{-\frac{d}{2}}(f_{+}(x) - f_{-}(x))$,
\item $f(\sqrt{2n}) = 0$ and $f'(\sqrt{2n}) \neq 0$,
\item $f(\sqrt{2m}) = f'(\sqrt{2m}) = 0$ for $m>n$.
\item $f(\sqrt{2m})$ are explicitly computable for  $0 \leq m < n$,
\end{itemize}
\end{Thm}
In order to use these functions for an application we would like to have better control over when exactly the double zeros begin.  For example, to achieve the best sphere packing bound we want to minimize $n$.  The following theorem shows that we have this control when $d \equiv 0 \pmod{8}$.
\begin{Thm} \label{8}
For each $d \equiv 0 \pmod{8}$, let $n_{+} = \left\lfloor \frac{d}{16} + \frac{1}{2} \right\rfloor$ and $n_{-}= \left\lfloor \frac{d}{16} + 1 \right\rfloor$.  Then there exists radial Schwartz functions $f_{\pm}: \R^{d} \to \R$ which satisfy
\begin{itemize}
\item $\widehat{f}_{\pm}(x) = \pm f_{\pm}(x)$. 
\item $f_{\pm}(\sqrt{2n_{\pm}}) =0$ and $f_{\pm}'(\sqrt{2n_{\pm}}) \neq 0$.
\item $f_{\pm}(\sqrt{2m}) = f_{\pm}'(\sqrt{2m}) =0$ for $m > n_{\pm}$.
\item $f_{\pm}(\sqrt{2m})$ are explicitly computable for $0 \leq m < n_{\pm}$. 
\end{itemize}
\end{Thm}
\begin{rmk}
From Proposition \ref{SZ} it is clear that values of $f_{\pm}(\sqrt{2m})$ fr $0 \leq m < n_{\pm}$ are explicitly computable from the coefficients of the modular forms that are used to construct these Schwartz functions.  In each case that we have computed these values are non-zero, and based on this computational evidence it appears that this is true in general.
\end{rmk}
\begin{rmk}
For any given $d$, it is straightforward to verify the inequalities needed in order to use the result of Cohn and Elkies.
\end{rmk}
\begin{rmk}
For $d=8$ and $d=24$ these constructions give the same functions as in \cite{V} and \cite{ckmrv}.
\end{rmk}
\begin{rmk}
After the necessary inequalities are checked, Theorem \ref{8} implies that 
\begin{equation*}
\Delta_{d} \leq 2^{-0.4529d} \qquad d \to \infty
\end{equation*}
for $d \equiv 0 \pmod{8}$. Although this falls short of the Kabatiansky-Levenshtein bound, it answers the question of what bounds can be obtained by results fitting the theory of Viazovska et al. in a natural family. It would be extremely interesting to search for a modular family yielding improved bounds, or to explore what bounds can be obtained in more general energy optimization problems via modular forms. 
\end{rmk}
\begin{rmk}
Henry Cohn has pointed out to the authors that for $d=12$ this construction gives the same function as in \cite{CG} which proves the optimal upper bound $A_{+}(12) = \sqrt{2}$ for the uncertainty principle of Bourgain, Clozel, and Kahane.  He also notes that, after checking some inequalities, this construction would give an upper bound of $A_{+}(28) \leq 2$ for $d=28$.  Cohn and Gon\c{c}alves showed in \cite{CG} that $A_{+}(28) <1.99$, but they conjecture that there exists an optimal function for $d=28$ with non-zero roots at radii $\sqrt{2j + o(1)}$ for $j \geq 2$.  
\end{rmk}
The Schwartz functions constructed by Viazovska, Cohn, Kumar, Miller, Radchenko, and Viazovska, and Cohn and Gon\c{c}alves in \cite{V}, \cite{ckmrv}, \cite{ckmrv2}, and \cite{CG} showed that there is a surprising and beautiful connection between modular forms and various problems related to sphere packing.  It is natural to seek a better understanding of these functions, and to aim for a general framework. In this spirit, here we show that Schwartz functions can be constructed  a la Viazovska in a uniform way, yielding the first natural family of functions extending those of Viazovska et al. We hope that this will be useful either for related problems of energy minimization or other methods suitable for applying linear programming bounds, or to inspire more work in hopes of a general framework.

The paper is organized as follows.
In Section 2 we will give a brief background on the modular forms necessary for the constructions of the Schwartz functions.  In Section 3 we will give the proof of Theorem \ref{Schwartz} by splitting up the construction into a ``plus" and ``minus" side and then showing how to control the zeros of these functions.  In Section 4 we will prove Theorem \ref{8} by studying the dimensions of certain spaces of modular forms.

\section*{Acknowledgements}
The authors would like to thank Ken Ono for his thoughts on an earlier version of this work and Henry Cohn for his many helpful comments which improved this paper.

\section{Background on modular forms}
We will begin with a review of classical modular forms.  Denote the \textit{upper half-plane} by $\mathbb{H} := \{z = x + iy \in \C: y>0 \}$.  The \textit{modular group}, denoted by $SL_{2}(\Z)$, is the group of $2 \times 2$ integer matrices with determinant one.  It is generated by the two elements
\begin{equation*}
T:= \begin{pmatrix} 1 & 1 \\ 0 & 1 \end{pmatrix} \ \text{ and } \ S:= \begin{pmatrix} 0 & -1 \\ 1 & 0 \end{pmatrix}.
\end{equation*}
An element $\gamma = \begin{pmatrix} a & b \\ c & d \end{pmatrix} \in SL_{2}(\Z)$ acts on a point $z \in \mathbb{H}$ by the M\"{o}bius transformation
\begin{equation*}
\gamma z := \frac{az + b}{cz + d}.
\end{equation*}
We also define the level two congruence subgroup
\begin{align*}
\Gamma(2) & := \left\{ \gamma \in SL_{2}(\Z): \gamma \equiv I \pmod{2} \right\},
\end{align*}
where $I = \begin{pmatrix} 1 & 0 \\ 0 & 1 \end{pmatrix}$.  The action of a congruence subgroup on $\HH$ extends to an action on $\Q \cup \{i \infty \}$.  A \textit{cusp} of a congruence subgroup $\Gamma$ is an equivalence class of $\Q \cup \{i \infty \}$ under the action of $\Gamma$.  For each integer $k$ and each $\gamma = \begin{pmatrix} a & b \\ c & d \end{pmatrix} \in SL_{2}(\Z)$ define the \textit{slash operator} on smooth functions $f: \HH \to \C$ by
\begin{equation*}
f |_{k} \gamma (z) := (cz+d)^{-k} f(\gamma z).
\end{equation*}
The slash operator satisfies $f|_{k}(\gamma_{1} \gamma_{2}) = (f|_{k} \gamma_{1})|_{k} \gamma_{2}$.  Letting $k$ be an integer, we require the following definition.
\begin{Definition}
Suppose $f: \HH \to \C$ is holomorphic.  Then $f(z)$ is a \textit{holomorphic modular form of weight $k$} on a congruence subgroup $\Gamma$ if 
\begin{enumerate}
\item $f |_{k} \gamma (z) = f(z)$ for all $\gamma \in \Gamma$.
\item $f(z)$ has at most polynomial growth at all the cusps of $\Gamma$.
\end{enumerate}
\end{Definition}  
Denote the space of weight $k$ holomorphic modular forms on $\Gamma$ by $M_{k}(\Gamma)$.  We can relax the definition of modular form to allow poles at the cusps.  Denote this space by $M_{k}^{!}(\Gamma)$.  Define the \textit{Eisenstein series} by
\begin{equation}
E_{k}(z) := \frac{1}{2 \zeta(k)} \sum_{(0,0) \neq (n,m) \in \Z^2} \frac{1}{(nz+m)^{k}} \in M_{k}(SL_{2}(\Z)),
\end{equation}
where $\zeta(s)$ is the Riemann zeta-function.  For $k>2$, $E_{k}(z)$ is absolutely convergent, and one can easily check the action of $S$ and $T$ to see it is modular on $SL_{2}(\Z)$.  For even $k>2$, its Fourier expansion with $q = e^{2 \pi i z}$ is given by
\begin{equation*}
E_{k}(z) = 1- \frac{2k}{B_{k}} \sum_{n \geq 1} \sigma_{k-1}(n) q^n,
\end{equation*}
where $B_{k}$ is the $k$th Bernoulli number and $\sigma_{k-1}(n)$ is the sum of divisors function given by
\begin{equation*}
\sigma_{k-1}(n) = \sum_{d \mid n} d^{k-1}.
\end{equation*}
For $k=2$, the Eisenstein series is no longer absolutely convergent.  One can still define the weight $2$ Eisenstein series by its Fourier expansion:
\begin{equation*}
E_{2}(z) = 1 -24 \sum_{n \geq 1} \sigma_{1}(n) q^n.
\end{equation*}
$E_{2}(z)$ is still periodic by definition, but has a slightly more complicated transformation under $S$, given by
\begin{equation*}
z^{-2} E_{2} \left(-\frac{1}{z} \right) = E_{2}(z) + \frac{6}{\pi i z}.
\end{equation*}    
The weight $2$ Eisenstein series is the first example of a \textit{quasi-modular form}.  We say that $f$ is a depth $d$ quasi-modular form if it is a degree $d$ polynomial in $E_{2}$ with modular form coefficients.  Another important modular form on $SL_{2}(\Z)$ is the weight $12$ Delta function
\begin{equation*}
\Delta(z) = \frac{E_{4}(z)^{3} - E_{6}(z)^2}{1728} = q\prod_{n=1}^{\infty}(1-q^n)^{24}.
\end{equation*}
The product formula shows that $\Delta(z)$ does not vanish on $\HH$.  Classical theory of modular forms implies that we have the following structure for algebras of modular forms, as graded rings:
\begin{align*}
M(SL_{2}(\Z)) &= \bigoplus_{k \in \Z} M_{k}(SL_{2}(\Z)) = \C [E_{4}, E_{6}], \\
M^{!}(\Gamma) &= \bigoplus_{k \in \Z} M_{k}^{!}(\Gamma) = M(\Gamma)[\Delta^{-1}].
\end{align*}

We also need the following classical \textit{Jacobi theta functions}
\begin{align*}
\theta_{2}(z) &= \sum_{n \in \Z} e^{\pi i \left( n + \frac{1}{2} \right)^{2} z}, \\
\theta_{3}(z) &= \sum_{n \in \Z} e^{\pi i n^2 z}, \\
\theta_{4}(z) &= \sum_{n \in \Z} (-1)^n e^{\pi i n^2 z},
\end{align*}
which are modular forms of weight $\frac12$ (for simplicity, we omit the exact definition of modular forms of half-integral weight).
Following the notation in \cite{ckmrv2} we consider the following special modular forms of weight $2$:
\begin{align} 
\begin{split}
U(z) &= \theta_{3}(z)^4 \\
V(z) &= \theta_{2}(z)^4 \\
W(z) &= \theta_{4}(z)^4.
\end{split}
\end{align}
With this notation we can write the Jacobi identity as $U = V + W$ and we have the fact
\begin{equation}
M(\Gamma(2)) = \C[V, W].
\end{equation}
The modular forms $U, V$, and $W$ transform under $SL_{2}(\Z)$ as follows:
\begin{alignat}{3}
U|_{2}T &= W, & \quad V|_{2} T &= -V, & \quad W|_{2} T &= U, \\
U|_{2} S &= -U, & \quad V|_{2} S &= -W, & \quad W|_{2} S &=-V.
\end{alignat}
We will also require the modular function
\begin{equation}
\lambda(z) := \frac{V}{U}(z) \in M_{0}^{!}(\Gamma(2)).
\end{equation}
The function $\lambda(s)$ is the {\it Hauptmodul} for $\Gamma(2)$, which means that it generates the function field for the modular curve (explicitly, $M_0^!(\Gamma(2))=\mathbb C(\lambda)$).  It takes the values $0, 1$, and $\infty$ at the cusps $i \infty, 0$, and $-1$ of $\Gamma(2)$ respectively, and it decreases from $1$ to $0$ as $z$ goes from $0$ to $i \infty$ along the imaginary axis.  The function $\lambda(z)$ satisfies the transformation properties
\begin{align}
\begin{split}
(\lambda|_{0} S) (z) &= 1- \lambda(z) \\
(\lambda|_{0}T) (z) &= -\frac{\lambda(z)}{1-\lambda(z)}.
\end{split}
\end{align}
If we define $\lambda_{S}(z) := (\lambda |_{0}S) (z)$, then we also have
\begin{equation}
(\lambda_{S} |_{0}T) (z) = \frac{1}{\lambda_{S}(z)}.
\end{equation}
We again follow \cite{ckmrv2} to define logarithms of $\lambda$ and $\lambda_{S}$.  Because $\lambda$ and $\lambda_{S}$ do not vanish on $\HH$ we can define
\begin{equation}
\mathcal{L}(z) := \int_{0}^{z} \frac{\lambda'(w)}{\lambda(w)} dw \quad \text{and} \quad \mathcal{L}_{s}(z) := - \int_{z}^{\i \infty} \frac{\lambda_{S}'(w)}{\lambda_{S}(w)} dw,
\end{equation}
where the contours are chosen to approach $0$ or $i \infty$ along vertical lines.  These functions are essentially the regularized Eichler integrals of the weight $2$ weakly holomorphic modular form $\frac{\lambda'(z)}{\lambda(z)}$ at the cusps $0$ and $i \infty$.  They therefore are the holomorphic parts of some weight $0$ harmonic Maass form and will play the same role for constructing Schwartz functions on the ``minus" side as $E_{2}$ plays on the ``plus" side.  For more information on these topics see \cite{BFOR}.  These functions satisfy
\begin{equation*}
\mathcal{L}(it) = \log(\lambda(it)) \quad \text{and} \quad \mathcal{L}_{S}(it) = \log(\lambda_{S}(it)) = \log(1 - \lambda(it))
\end{equation*}
for $t>0$, and so are holomorphic functions for which $e^{\mathcal{L}} = \lambda$ and $e^{\mathcal{L}_{S}}= \lambda_{S}$, but are not in general the principal branches of the logarithms of $\lambda$ and $\lambda_{S}$.  We have the following asymptotics as $q \to 0$:
\begin{align}
\begin{split}
\mathcal{L}(z) &= \pi i z + 4 \log(2) -8 q^{\frac{1}{2}} + O(q) \\
\mathcal{L}_{S}(z) &= -16q^{\frac{1}{2}} - \frac{64}{3} q^{\frac{3}{2}} + O(q^{\frac{5}{2}}).
\end{split}
\end{align}
The functions $\mathcal{L}$ and $\mathcal{L}_{S}$ satisfy the transformation properties
\begin{alignat}{2} \label{L}
\mathcal{L}|_{0} T^{\pm 1} &= \mathcal{L} - \mathcal{L}_{S} \pm \pi i, \quad & \mathcal{L}_{S}|_{0} T &= -\mathcal{L}_{S}, \\
\mathcal{L}|_{0} S &= \mathcal{L}_{S}, \quad & \mathcal{L}_{S}|_{0}S &= \mathcal{L},
\end{alignat}
where $f|_{k}T^{-1} = f(z-1)$.

\section{Proof of Theorem \ref{Schwartz}}
\subsection{The $+1$ eigenfunction construction}

In this section we discuss generalizations of Viazovska's $+1$ eigenfunction construction.  Let
\begin{equation*}
\phi(z) = \sum c_{\phi}(n) q^{n},
\end{equation*}
be a $1$-periodic function on the upper half-plane.  The following proposition presents our function of interest in a form where its Fourier transform is easily calculable.
\begin{Prop} \label{S1}
Let $\phi(z)$ be a $1$-periodic function that vanishes as $z \to i \infty$ and suppose there is an $r_{0} \geq 0$ such that
\begin{align*}
\phi \left(\frac{i}{t} \right) &= O \left(t^{-\frac{d}{2} +2} e^{r_{0}^2 \pi t} \right) \qquad t \to \infty.
\end{align*}
Then for $ x \in \R^{d}$ 
\begin{align*}
a(x) &:= \int_{-1}^{i} \phi \left(-\frac{1}{z+1} \right) (z+1)^{\frac{d}{2} -2} e^{\pi i |x|^{2} z} dz +  \int_{1}^{i} \phi \left(-\frac{1}{z-1} \right) (z-1)^{\frac{d}{2} -2} e^{\pi i |x|^{2} z} dz \\
& -2\int_{0}^{i} \phi \left(-\frac{1}{z} \right) z^{\frac{d}{2} -2} e^{\pi i |x|^{2} z} dz +2\int_{i}^{i \infty} \phi \left(z \right) e^{\pi i |x|^{2} z} dz
\end{align*}
is a radial Schwartz function and $\widehat{a}(x) = (-i)^{-\frac{d}{2}} a(x)$.
\end{Prop}
\begin{proof}
By hypothesis,  $\phi(z)$ decays exponentially as ${\rm{Im}}(z) \to \infty$, all of the above terms will be bounded and $a$ and all of its derivatives will decay exponentially so $a$ is Schwartz.
Because the integrals are absolutely and uniformly convergent we can switch the order of the integrals to compute:
\begin{align*}
\widehat{a}(x) &=  \int_{-1}^{i} \phi \left(-\frac{1}{z+1} \right) (z+1)^{\frac{d}{2} -2} (-iz)^{-\frac{d}{2}} e^{\pi i |x|^{2} \left(-\frac{1}{z} \right)} dz \\
&+ \int_{1}^{i} \phi \left(-\frac{1}{z-1} \right) (z-1)^{\frac{d}{2} -2} (-iz)^{-\frac{d}{2}} e^{\pi i |x|^{2} \left(-\frac{1}{z} \right)} dz \\
& -2\int_{0}^{i} \phi \left(-\frac{1}{z} \right) z^{\frac{d}{2} -2} (-iz)^{-\frac{d}{2}} e^{\pi i |x|^{2} \left(-\frac{1}{z} \right)} dz +2\int_{i}^{i \infty} \phi \left(z \right) (-iz)^{-\frac{d}{2}} e^{\pi i |x|^{2} \left(-\frac{1}{z} \right)} dz.
\end{align*}
Letting $w = -\frac{1}{z}$, we find:
\begin{align*}
\widehat{a}(x) &= (-i)^{-\frac{d}{2}} \int_{1}^{i} \phi \left(1-\frac{1}{w-1} \right) \left(1- \frac{1}{w}\right)^{\frac{d}{2} -2} w^{\frac{d}{2} -2} e^{\pi i |x|^{2} w} dw \\
&+ (-i)^{-\frac{d}{2}} \int_{-1}^{i} \phi \left(1-\frac{1}{w+1} \right) \left(-1- \frac{1}{w}\right)^{\frac{d}{2} -2} w^{\frac{d}{2} -2} e^{\pi i |x|^{2} w} dw \\
& -2(-i)^{-\frac{d}{2}} \int_{i \infty}^{i} \phi \left(w \right) e^{\pi i |x|^{2} w} dw +2(-i)^{-\frac{d}{2}} \int_{i}^{0} \phi \left(-\frac{1}{w} \right) w^{\frac{d}{2} -2} e^{\pi i |x|^{2} w} dw \\
&= (-i)^{-\frac{d}{2}} \int_{1}^{i} \phi \left(-\frac{1}{w-1} \right) \left(w- 1 \right)^{\frac{d}{2} -2} e^{\pi i |x|^{2} w} dw \\
&+(-i)^{-\frac{d}{2}} \int_{-1}^{i} \phi \left(-\frac{1}{w+1} \right) \left(-w- 1 \right)^{\frac{d}{2} -2} e^{\pi i |x|^{2} w} dw \\
& +2(-i)^{-\frac{d}{2}} \int_{i}^{i \infty} \phi \left(w \right) e^{\pi i |x|^{2} w} dw -2(-i)^{-\frac{d}{2}} \int_{0}^{i} \phi \left(-\frac{1}{w} \right) w^{\frac{d}{2} -2} e^{\pi i |x|^{2} w} dw \\
&= (-i)^{-\frac{d}{2}}a(x).
\end{align*}
Note that the only property we used above is that $\phi(z)$ is $1$-periodic.
\end{proof}

In her work, Viazovska used special choices of functions $\phi$ to show that the resulting $a(x)$ has the additional property that it has double zeros at vectors of length $\sqrt{2k}$, for $k>1$ and $k>2$, and a single zero at vectors of length $\sqrt{2}$ and $2$ in dimensions $8$ and $24$ respectively.  The significance of this is that the former numbers are the non-minimal length vectors in the $E_{8}$ and Leech lattice respectively.  Her idea was to relate $a(r)$ satisfying the hypothesis in the proposition above to a function with these specific zeros.  The asymptotic behavior of the $\phi$ combined with the simple characterization zeros of the $\sin^{2}$ factor in the next proposition offers this description.   
\begin{Prop} \label{S2}
Suppose that $\phi(z)$ is a weakly holomorphic quasi-modular form of weight $k=-\frac{d}{2} +4$ and depth $2$ on $SL_{2}(\Z)$ satisfying the conditions of Proposition \ref{S1}. Then if $r \geq r_{0}$ we have that
\begin{equation*}
a(r) = -4\sin^{2} \left( \frac{\pi r^2}{2} \right) \int_{0}^{i \infty} \phi \left( - \frac{1}{z} \right) z^{\frac{d}{2} -2} e^{\pi i r^2 z} dz.
\end{equation*}

\end{Prop}

\begin{proof}
By direct calculation we have that
\begin{align*}
 &-4\sin^{2} \left( \frac{\pi r^2}{2} \right) \int_{0}^{i \infty} \phi \left( - \frac{1}{z} \right) z^{\frac{d}{2} -2} e^{\pi i r^2 z} dz \\
 &= \int_{0}^{i \infty} \phi \left( - \frac{1}{z} \right) z^{\frac{d}{2} -2} e^{\pi i r^2 (z+1)} dz - 2\int_{0}^{i \infty} \phi \left( - \frac{1}{z} \right) z^{\frac{d}{2} -2} e^{\pi i r^2 z} dz \\
 &+ \int_{0}^{i \infty} \phi \left( - \frac{1}{z} \right) z^{\frac{d}{2} -2} e^{\pi i r^2 (z-1)} dz \\
&= \int_{1}^{i \infty +1} \phi \left( - \frac{1}{z-1} \right) (z-1)^{\frac{d}{2} -2} e^{\pi i r^2 z} dz - 2\int_{0}^{i \infty} \phi \left( - \frac{1}{z} \right) z^{\frac{d}{2} -2} e^{\pi i r^2 z} dz \\
&+\int_{-1}^{i \infty -1} \phi \left( - \frac{1}{z+1} \right) (z+1)^{\frac{d}{2} -2} e^{\pi i r^2 z} dz.
\end{align*}
We can deform the path of integration because the integrand decays as ${\rm{Im}}(z) \to \infty$ to see:
\begin{align*}
&\int_{1}^{i} \phi \left( - \frac{1}{z-1} \right) (z-1)^{\frac{d}{2} -2} e^{\pi i r^2 z} dz - 2\int_{0}^{i} \phi \left( - \frac{1}{z} \right) z^{\frac{d}{2} -2} e^{\pi i r^2 z} dz \\
&+\int_{-1}^{i} \phi \left( - \frac{1}{z+1} \right) (z+1)^{\frac{d}{2} -2} e^{\pi i r^2 z} dz \\
&+ \int_{i}^{i \infty} \left[ \phi \left( - \frac{1}{z-1} \right) (z-1)^{\frac{d}{2} -2} + \phi \left( - \frac{1}{z+1} \right) (z+1)^{\frac{d}{2} -2} -2\phi \left( - \frac{1}{z} \right) z^{\frac{d}{2} -2} \right] e^{\pi i r^{2} z} dz.
\end{align*}
By using the transformation properties of a depth $2$ quasi-modular form we find that this last expression is $a(r)$.

\end{proof}

\subsection{The $-1$ eigenfunction construction}

In the previous section we discussed the method Viazovska used to construct Schwartz functions that were eigenfunctions of the Fourier transform with eigenvalue $+1$.  Viazovska also used theta functions to construct Schwartz functions with eigenvalue $-1$ under the Fourier transform.  Here we generalize this by studying weakkly holomorphic modular forms on $\Gamma(2)$.  For a modular form $\psi(z) \in M_{k}^{!}(\Gamma(2))$, let $\psi_{\gamma}(z):= \psi_{I}(z) |_{k} \gamma$.

\begin{Prop} \label{S3}
Let $\psi_{I}(z)$ be a weight $-\frac{d}{2} +2$ weakly holomorphic modular form on $\Gamma(2)$ that vanishes as $z \to 0$ and suppose that there is an $r_{0} \geq 0$ such that 
\begin{align*}
\psi_{I}(it) &= O(e^{r_{0}^{2} \pi t}) \qquad t \to \infty \\
\psi_{I}(z) &= \psi_{T}(z) + \psi_{S}(z).
\end{align*}
Then for $x \in \R^{d}$,
\begin{align*}
b(x)&:= \int_{-1}^{i} \psi_{T}(z) e^{\pi i |x|^{2}z} dz + \int_{1}^{i} \psi_{T}(z) e^{\pi i |x|^{2}z} dz \\
&-2\int_{0}^{i} \psi_{I}(z) e^{\pi i |x|^{2}z} dz - 2 \int_{i}^{i \infty} \psi_{S}(z) e^{\pi i |x|^{2}z} dz.
\end{align*}
is a radial Schwartz function and $\widehat{b}(x) =-(-i)^{-\frac{d}{2}} b(x)$.
\end{Prop}
\begin{proof}
The fact that $b(x)$ is a Schwartz functions follows the same way as before.  The Fourier transform of $b(x)$ is given as
\begin{align*}
\widehat{b}(x) &= \int_{-1}^{i} \psi_{T}(z) (-iz)^{-\frac{d}{2}} e^{\pi i |x|^{2}\left(-\frac{1}{z} \right)} dz + \int_{1}^{i} \psi_{T}(z) (-iz)^{-\frac{d}{2}}  e^{\pi i |x|^{2}\left(-\frac{1}{z} \right)} dz \\
&-2\int_{0}^{i} \psi_{I}(z) (-iz)^{-\frac{d}{2}}  e^{\pi i |x|^{2}\left(-\frac{1}{z} \right)} dz - 2 \int_{i}^{i \infty} \psi_{S}(z) (-iz)^{-\frac{d}{2}}  e^{\pi i |x|^{2}\left(-\frac{1}{z} \right)} dz.
\end{align*}
We substitute $w = -\frac{1}{z}$ as before and use the facts
\begin{align*}
\psi_{T} \left(-\frac{1}{z} \right) &= - \psi_{T}(z) z^{-\frac{d}{2} +2}, \\
\psi_{I} \left(-\frac{1}{z} \right) &= \psi_{S}(z) z^{-\frac{d}{2} +2}
\end{align*}
to show that $\widehat{b}(x) = -(-i)^{-\frac{d}{2}}b(x)$.
\end{proof}
Following the same ideas as for $a(x)$, we have
\begin{Prop} \label{S4}
Suppose that $\psi_{I}(z)$ is a weakly holomorphic modular form of weight $-\frac{d}{2} +2$ on $\Gamma(2)$ satisfying the conditions of Proposition \ref{S3}.  Then if $r \geq r_{0}$ we have that
\begin{equation*}
b(r) = -4 \sin^{2} \left(\frac{\pi r^2}{2} \right) \int_{0}^{i \infty} \psi_{I}(z) e^{\pi i r^2 z} dz.
\end{equation*}
\end{Prop}
\begin{proof}
The proof follows almost the same as the proof for Proposition \ref{S2}.  The main points we use to show this are that
\begin{equation*}
\psi_{I}(z-1) = \psi_{I}(z+1) = \psi_{T}(z)
\end{equation*}
and 
\begin{equation*}
\psi_{T}(z) - \psi_{I}(z) = -\psi_{S}(z).
\end{equation*}
\end{proof}
The following propositions generalize Proposition \ref{S3} and Proposition \ref{S4} to allow us to use $\mathcal{L}(z)$.  As we will explain in Section 4, this construction was not needed to resolve the sphere packing problem in dimensions $8$ and $24$, but allows better control over $n_{-}$ in general in Theorem \ref{8}.
\begin{Prop} \label{SL1}
Let $g(z) = f(z) \mathcal{L}(z)$ where $f(z)$ is a weight $-\frac{d}{2} +2$ weakly holomorphic modular form on $SL_{2}(\Z)$.  Suppose $g(z)$ vanishes as $z \to 0$ and that there is an $r_{0} \geq 0$ such that
\begin{equation*}
g(it) = O(t e^{r_{0}^2 \pi t}) \quad t \to \infty.
\end{equation*}
Then for $x \in \R^{d}$,
\begin{align*}
c(x)&:= \int_{-1}^{i} g_{T}(z) e^{\pi i |x|^{2}z} dz + \int_{1}^{i} g_{T^{-1}}(z) e^{\pi i |x|^{2}z} dz \\
&-2\int_{0}^{i} g(z) e^{\pi i |x|^{2}z} dz - 2 \int_{i}^{i \infty} g_{S}(z) e^{\pi i |x|^{2}z} dz.
\end{align*}
is a radial Schwartz function and $\widehat{c}(x) =-(-i)^{-\frac{d}{2}} c(x)$.
\end{Prop}
\begin{proof}
As before we have
\begin{align*}
\widehat{c}(x) &= \int_{-1}^{i} g_{T}(z) (-iz)^{-\frac{d}{2}} e^{\pi i |x|^{2}\left(-\frac{1}{z} \right)} dz + \int_{1}^{i} g_{T^{-1}}(z) (-iz)^{-\frac{d}{2}}  e^{\pi i |x|^{2}\left(-\frac{1}{z} \right)} dz \\
&-2\int_{0}^{i} g(z) (-iz)^{-\frac{d}{2}}  e^{\pi i |x|^{2}\left(-\frac{1}{z} \right)} dz - 2 \int_{i}^{i \infty} g_{S}(z) (-iz)^{-\frac{d}{2}}  e^{\pi i |x|^{2}\left(-\frac{1}{z} \right)} dz.
\end{align*}
Let $w = -\frac{1}{z}$ to arrive at
\begin{align*}
\widehat{c}(x) &= (-i)^{-\frac{d}{2}} \int_{1}^{i} g_{T} \left( -\frac{1}{w} \right) w^{\frac{d}{2} -2} e^{\pi i |x|^2 w} dw + (-i)^{-\frac{d}{2}} \int_{-1}^{i} g_{T^{-1}} \left( -\frac{1}{w} \right) w^{\frac{d}{2} -2} e^{\pi i |x|^2 w} dw \\
& -2(-i)^{-\frac{d}{2}} \int_{i \infty}^{i} g \left( -\frac{1}{w} \right) w^{\frac{d}{2} -2} e^{\pi i |x|^2 w} dw -2(-i)^{-\frac{d}{2}} \int_{i}^{0} g_{S} \left( -\frac{1}{w} \right) w^{\frac{d}{2} -2} e^{\pi i |x|^2 w} dw.
\end{align*}
By using the transformation properties of $\mathcal{L}$ given in equation \eqref{L} we have that
\begin{align*}
g_{T}|_{-\frac{d}{2} +2} S &= -g_{T^{-1}} \\
g_{T^{-1}}|_{-\frac{d}{2} +2} S &= -g_{T}.
\end{align*}
Using these properties it is clear to see that $\widehat{c}(x) =-(-i)^{-\frac{d}{2}} c(x)$.
\end{proof}

In analogy with the previous propositions we have the following.
\begin{Prop} \label{SL2}
Suppose that $g(z)$ is as in Proposition \ref{SL1}.  Then if $r \geq r_{0}$ we have that
\begin{equation*}
c(r) = -4 \sin^{2} \left(\frac{\pi r^2}{2} \right) \int_{0}^{i \infty} g(z) e^{\pi i r^2 z} dz.
\end{equation*}
\end{Prop}
\begin{proof}
By direct calculation we have
\begin{align*}
& -4 \sin^{2} \left(\frac{\pi r^2}{2} \right) \int_{0}^{i \infty} g(z) e^{\pi i r^2 z} dz \\
&= \int_{0}^{i \infty} g(z) e^{\pi i r^2 (z+1)} dz -2\int_{0}^{i \infty} g(z) e^{\pi i r^2 z} dz + \int_{0}^{i \infty} g(z) e^{\pi i r^2 (z-1)} dz \\
&= \int_{1}^{i \infty} g_{T^{-1}}(z) e^{\pi i r^2 z} dz -2\int_{0}^{i \infty} g(z) e^{\pi i r^2 z} dz + \int_{-1}^{i \infty} g_{T}(z) e^{\pi i r^2 z} dz.
\end{align*}
The integrand decays as $z \to i \infty$ so we can deform the path of integration to arrive at
\begin{align*}
&\int_{1}^{i} g_{T^{-1}}(z) e^{\pi i r^2 z} dz -2\int_{0}^{i} g(z) e^{\pi i r^2 z} dz + \int_{-1}^{i} g_{T}(z) e^{\pi i r^2 z} dz \\
&+ \int_{i}^{i \infty} (g_{T^{-1}}(z) + g_{T}(z) -2g(z)) e^{\pi i r^2 z} dz.
\end{align*}
By the properties of $\mathcal{L}$ given in equation \eqref{L} we have 
\begin{equation*}
g_{T^{\pm 1}} = f \mathcal{L}|_{-\frac{d}{2} +2} T^{\pm 1} = f( \mathcal{L} - \mathcal{L}_{S} \pm \pi i).
\end{equation*} 
From this it is clear that $g_{T} + g_{T^{-1}} =2g - 2g_{S}$.  Using this transformation property completes the proof.
\end{proof}
\subsection{The zeros of the Schwartz functions}
The following proposition gives the conditions needed to control the location of the simple zero and when the double zeros begin for the Schwartz functions.
\begin{Prop} \label{SZ}
Assume that the minimal length vector of the lattice of interest has the form $r_{0} = \sqrt{2k}$ for some $k \in \Z$.  If 
\begin{equation*}
g(z) = p(z) +O \left(z^2 e^{2 \pi iz} \right)
\end{equation*}
with
\begin{equation*}
p(z)= c_{0}e^{-r_{0}^{2} \pi i z} + c_{1}ze^{-(r_{0}^{2} -2) \pi iz} + c_{2}e^{-(r_{0}^{2} -2) \pi iz} + \cdots + c_{2k-1}z + c_{2k}
\end{equation*}
where the $c_{j}$ are constants and $c_{0}, c_{2m-1} \neq 0$ for $1 \leq m \leq k$, then if
\begin{equation*}
f(r) = -4\sin^{2} \left( \frac{\pi r^2}{2} \right) \int_{0}^{i \infty} g(z) e^{\pi i r^2 z} dz
\end{equation*}
\begin{align*}
f \left(\sqrt{2k} \right) &= f(r_{0}) = 0, \\
f' \left(\sqrt{2k} \right) &= f'(r_{0}) \neq 0, \\
f \left(\sqrt{2m} \right) &\neq 0 \qquad 0 \leq m \leq k-1.
\end{align*}
\end{Prop}

\begin{proof}
If we make the substitution $z = it$ then
\begin{equation*}
f(r) = -i^{\frac{d}{2} -1} 4 \sin^{2} \left( \frac{\pi r^{2}}{2} \right) \left[\int_{0}^{\infty} p(it) e^{-\pi r^{2}t} dt + \int_{0}^{\infty} \left( g(it)- p(it) \right)e^{-\pi r^{2}t} dt \right].
\end{equation*}
We have that
\begin{align*}
\int_{0}^{\infty} &p(it) e^{-\pi r^{2} t} dt = \int_{0}^{\infty} \left(c_{0} e^{r_{0}^{2} \pi t} + ic_{1} t e^{(r_{0}^{2} -2) \pi t} + \cdots + i c_{2k-1}t + c_{2k} \right) e^{-\pi r^{2} t} dt \\
&= \frac{c_{0}}{\pi (r^{2} - r_{0}^{2})} + \frac{ic_{1}}{\pi^{2} (r^{2} - r_{0}^{2} +2)^{2}} + \frac{c_{2}}{\pi (r^2 - r_{0}^{2} +2)} + \cdots + \frac{ic_{2k-1}}{\pi^{2} r^{4}} + \frac{c_{2k}}{\pi r^2}.
\end{align*}
When this term is multiplied by $\sin^{2} \left( \frac{\pi r^{2}}{2} \right)$ it is clear that we get a zero at $r=r_{0}=\sqrt{2k}$ and that $a \left( \sqrt{2m} \right) \neq 0$ form $0 \leq m \leq k-1$.  The first term also ensures that the zero at $r=r_{0}$ only has order one.  It is also clear that $a(r)$ has double zeros at $r= \sqrt{2m}$ for $m > k$.

\end{proof}
To use this for the $+1$ eigenfunction we replace $g(z)$ by $\phi \left(-\frac{1}{z} \right) z^{\frac{d}{2}-2}$.  To use it for the $-1$ eigenfunction we replace $g(z)$ by $\psi(z)$.

\section{Proof of Theorem \ref{8}}
\subsection{The $+1$ eigenfunction}
In this section we will study when it is possible to construct the $+1$ eigenfunctions.  Let $d \equiv 0 \pmod{8}$.  We can assume that our quasi-modular form $\phi(z)$ is always a holomorphic quasi-modular form divided by some power of $\Delta(z)$.  The conditions given above are equivalent to demanding that
\begin{equation*}
\widetilde{\phi}(z) = \Delta^{n}(z) \phi(z)
\end{equation*}
is a weight $-\frac{d}{2} +4 +12n$ quasi-modular form of depth $2$ on $SL_{2}(\Z)$ such that $\widetilde{\phi}(z) = O(q^{n+1})$ with $n$ minimum.  All such forms are of the form
\begin{equation*}
\widetilde{\phi}(z) = \sum_{i \geq 1} \alpha_{i} E_{2}^{a_{i}}(z) E_{4}^{b_{i}}(z) E_{6}^{c_{i}}(z)
\end{equation*}
with atleast one $a_{i} =2$, all $a_{i} \leq 2$, and $2a_{i} + 4b_{i} + 6c_{i} = -\frac{d}{2} +4 +12n$ for all $i$.  Equivalently
\begin{equation*}
\widetilde{\phi}(z) \in E_{2}^2 M_{-\frac{d}{2} + 12n}(SL_{2}(\Z)) \bigoplus E_{2} M_{-\frac{d}{2} +2 + 12n}(SL_{2}(\Z)) \bigoplus M_{-\frac{d}{2} + 4 +12n}(SL_{2}(\Z)).
\end{equation*}
The number of such forms is 
\begin{equation*}
\delta_{d,n} := {\rm{dim}}\left(M_{-\frac{d}{2} +4 + 12n}(SL_{2}(\Z)) \right) + {\rm{dim}}\left(M_{-\frac{d}{2} +2 + 12n}(SL_{2}(\Z))\right) + {\rm{dim}}\left(M_{-\frac{d}{2} + 12n}(SL_{2}(\Z))\right),
\end{equation*}
and  it is well-known that
\begin{equation*}
{\rm{dim}}(M_{k}(SL_{2}(\Z))) = \begin{cases} \left\lfloor \frac{k}{12} \right\rfloor +1 & k \not\equiv 2 \pmod{12} \\
\left\lfloor \frac{k}{12} \right\rfloor & k \equiv 2 \pmod{12}. \end{cases} \end{equation*}
A short calculation shows that $\delta_{d,n} = 3n-\frac{d}{8} +2$.  In order to ensure that $\widetilde{\phi}(z) = O(q^{n+1})$ there needs to be a nontrivial solution to a system of $n+1$ homogeneous equations with $3n-\frac{d}{8}+2$ variables.  Therefore, we must have $2n>\frac{d}{8} -1$.
\begin{Ex}
For $d=8$ we can let $n=1$ and find that $\widetilde{\phi}(z) = E_{2}^2 E_{4}^2 - 2E_{2}E_{4}E_{6} + E_{6}^2$ which matches the function found in \cite{V}. 

For $d=48$ we can let $n=3$ and find 
\begin{align*}
\widetilde{\phi}(z)&=\Delta^{3}(z) \phi(z)= 1556796748E_{2}^{2}(z) E_{4}^{3}(z) -77235475E_{2}^{2}(z)E_{6}^{2}(z) \\
&-704733786E_{2}(z) E_{4}^{2}(z) E_{6}(z) -1029088507E_{4}^{4}(z) +254261020E_{4}(z) E_{6}^{2}(z) \\
 &= -1673465440313507328q^{4} +O(q^5).
\end{align*}
Dimension $d=48$ is especially interesting as the bound given by the $+1$ eigenfunction in this case exactly matches the lower bound given by the even unimodular lattice $P_{48n}$.
\end{Ex}

\subsection{The $-1$ eigenfunction}
We will follow the same basic argument as in the previous section.  Proposition \ref{S3} and Proposition \ref{SL1} show that the modular function, $\psi(z)$, we use to construct our Schwartz function must be a sum of a modular form $g$ of weight $-\frac{d}{2}+2$ on $\Gamma(2)$ such that $g=g_{T}+g_{S}$ and a function of the form $f \mathcal{L}$ where $f$ is a modular form of weight $-\frac{d}{2}+2$ on $SL_{2}(\Z)$.  In \cite{ckmrv2} it was shown that this is equivalent to
\begin{equation*}
\widetilde{\psi}(z) \in \left(U^2 - V^2 \right) M_{-\frac{d}{2}-2+12n}(SL_{2}(\Z)) \bigoplus WM_{-\frac{d}{2} +12n}(SL_{2}(\Z)) \bigoplus \mathcal{L} M_{-\frac{d}{2}+2 +12n}(SL_{2}(\Z)),
\end{equation*}
where $\widetilde{\psi}(z) = \Delta^{n}(z) \psi(z)$.  We now want to choose a $\widetilde{\psi}(z)$ in this space such that $\widetilde{\psi}(z)$ has a constant term without a $z$.  This ensures that the Schwartz function will have a simple zero at $\sqrt{2n}$.  We also need to ensure $\widetilde{\psi}_{S}(z) = O(q^{n + \frac{1}{2}})$ so that that $\psi(z)$ vanishes as $z \to 0$.  $\widetilde{\psi}_{S}(z)$ is only supported on half-integral exponents so this gives a system of $n+2$ homogeneous equations.  Let
\begin{equation*}
\delta_{d,n}':={\rm{dim}}\left(M_{-\frac{d}{2} -2 + 12n}(SL_{2}(\Z)) \right) + {\rm{dim}}\left(M_{-\frac{d}{2}+ 12n}(SL_{2}(\Z))\right) + {\rm{dim}}\left(M_{-\frac{d}{2} +2+ 12n}(SL_{2}(\Z))\right),
\end{equation*}
then a short computation shows $\delta_{d,n}'= 3n - \frac{d}{8} +1$ and so to guarantee a nontrivial solution we must have $2n \geq \frac{d}{8} +1$.
\begin{rmk}
One can ignore the contribution from $\mathcal{L}$ for $d=8$ and $d=24$ and get the same minimal value for $n$.  For example, for $d=8$ using the method described above we find
\begin{equation*}
\widetilde{\psi}(z) = \frac{1}{3} \left(U^2 - V^2 \right) E_{6} + \frac{2}{3} W E_{4}^2,
\end{equation*} which is equal to the form used in \cite{V}.  For this reason $\mathcal{L}$ did not show up in the constructions in \cite{V} or \cite{ckmrv}.
\end{rmk}
\begin{rmk}
This also shows that for $d=48$ the minimal possible $n$ is $n=4$.  Therefore, one cannot match the function found on the ``plus" side and resolve the sphere packing problem for $d=48$ using this method.
\end{rmk}
The following table gives the sphere packing upper bounds for some dimensions obtained using the functions constructed here.  

\begin{center}
\begin{tabular}{ |c|c|c|c| } 
 \hline
d & Given upper bound & Best upper bound & Best lower bound \\
\hline
$8$ & $ 0.2537$ & $0.2537$ & $ 0.2537$ \\
\hline
$16$ & $ 0.23533$ & $0.02519$ & $0.01471$ \\
\hline
$24$ & $0.0019$ & $0.0019$ & $0.0019$ \\
\hline
$48$ & $ 2.310 \times 10^{-5}$ & $4.529 \times 10^{-7}$ & $2.318 \times 10^{-8}$ \\
\hline
$72$ & $4.495 \times 10^{-10}$ & $6.736 \times 10^{-11}$ & $1.459 \times 10^{-13}$ \\
\hline
$96$ & $7.666 \times 10^{-12}$ & $7.945 \times 10^{-15}$ & $2.795 \times 10^{-22}$ \\
 \hline
\end{tabular}
\end{center}


\begin{thebibliography}{9}

\bibitem{BFOR}
K.~Bringmann, A.~Folsom, K.~Ono, and L.~Rolen.
\newblock Harmonic Maass forms and mock modular forms: theory and applications.
\newblock Published by the American Mathematical Society, Providence, RI, 2017.


\bibitem{CE}
H.~Cohn and N.~Elkies.
\newblock New upper bounds on sphere packings I.
\newblock Annals of Math. 157, 689-714, 2003.

\bibitem{CG}
H.~Cohn and F.~Gon\c{c}alves.
\newblock An optimal uncertainty principle in twelve dimensions via modular forms.
\newblock arXiv:1712.04438

\bibitem{CK}
H.~Cohn and A.~Kumar.
\newblock Universally optimal distribution of points on spheres.
\newblock Journal of Amer. Math. Soc. 20, 99-148, 2007.

\bibitem{ckmrv}
H.~Cohn, A.~Kumar, S.~Miller, D.~Radchenko, and M.~Viazovska.
\newblock The sphere packing problem in dimension $24$.
\newblock Annals of Mathematics, 185, 1017-1033, 2017.

\bibitem{ckmrv2}
H.~Cohn, A.~Kumar, S.~Miller, D.~Radchenko, and M.~Viazovska.
\newblock Universal optimality of the $E_{8}$ and Leech lattices and interpolation formulas.
\newblock arXiv:1902.05438

\bibitem{H}
T.~Hales.
\newblock A proof of the Kepler conjecture.
\newblock Annals of Mathematics, 162, 1065-1185, 2005.


\bibitem{KL}
G.A.~Katabiansky and V.I.~Levenshtein.
\newblock Bounds for packings on a sphere and in space.
\newblock Problemy Pereda\u{c}i Informacii 14, 3-25, 1978.
\newblock English translation in Problems of Information Transmission 14, 1-17, 1978.

\bibitem{Th}
A.~Thue.
\newblock \"{U}ber die dichteste Zusammenstellung von kongruenten Kreisen in einer
Ebene.
\newblock Norske Vid. Selsk. Skr. No.1 (1910), pp. 1-9.

\bibitem{Ve}
A.~Venkatesh.
\newblock A note on sphere packings in high dimension.
\newblock International Math. Res. Notices, 2013, 1624-1648, 2013.


\bibitem{V}
M.~Viazovska.
\newblock The sphere packing problem in dimension $8$.
\newblock Annals of Mathematics, 185, 991-1015, 2017.


\end{thebibliography}
\end{document}